\documentclass{amsart}
\usepackage{amsmath,amsfonts,amsthm}
\usepackage{amssymb,latexsym}
\usepackage{graphics}
\usepackage[colorlinks]{hyperref}

 \usepackage[all]{xypic}
\usepackage{amssymb}

\theoremstyle{plain}
\theoremstyle{definition}
\newtheorem{theorem}{Theorem}%[section]
\newtheorem{lemma}[theorem]{Lemma}

\newtheorem{corollary}[theorem]{Corollary}

\newtheorem{definition}[theorem]{Definition}
\newtheorem{example}[theorem]{Example}
\newtheorem{counterexample}[theorem]{Counterexample}

\newtheorem{note}[theorem]{Note}

\newtheorem{convention}[theorem]{Convention}
\newtheorem{remark}[theorem]{Remark}
\theoremstyle{remark}

\numberwithin{equation}{section}
\newcommand{\SP}{\: \: \: \: \:}

\title[Topological stability and shadowing  modulo an ideal]{Topological Stability and Shadowing in Uniform Transformation Semigroups Modulo an Ideal}
\author[F. Ayatollah Zadeh Shirazi, E. Hakimi, A. Hosseini, Kh. Tajbakhsh]{Fatemah Ayatollah Zadeh Shirazi, Elaheh Hakimi \\ Arezoo Hosseini,  Khosro Tajbakhsh}
\begin{document}
%%%%%%%%%%%%%%%%%%%%%%%%%%%%%%%%%%%% abstract
\begin{abstract}

In this paper, we introduce and analyze several key dynamical properties—namely shadowing modulo an ideal, expansivity modulo an ideal, and topological stability modulo an ideal—within the framework of uniform transformation semigroups. Given an ideal $\mathcal I$ on semigroup $T$, we investigate the interplay between these properties in compact Hausdorff transformation semigroup $(T,X,\mathfrak{X})$. Our main result establishes that if a compact Hausdorff transformation semigroup exhibits the shadowing property modulo $\mathcal I$ and is expansive modulo $\mathcal I$ then it is also topologically stable modulo $\mathcal I$. This extends known stability theorems in classical topological dynamics to the setting of ideal-constrained dynamics. Additionally, we explore the relationship between shadowing modulo $\mathcal I$ and the conventional shadowing property.
\end{abstract}
\maketitle
%%%%%%%%%%%%%%%%%%%%%%%%%%%%%%%%%%%% MSC
\noindent {\small {\bf 2020 Mathematics Subject Classification:} 54H15, 37B52 \\
{\bf Keywords:}} Ideal, Shadowing, Topological stability, Transformation semigroup, Uniform space.
%%%%%%%%%%%%%%%%%%%%%%%%%%%%%%%%%%%%
\section{Introduction}
\noindent The study of shadowing properties and their relationship with topological stability has been a central topic in dynamical systems research, as evidenced by several key works \cite{lee, pilyugin, walters, yan, yang}. Shadowing, which describes the ability of a system to approximate pseudo-orbits with true orbits, plays a crucial role in understanding the stability and robustness of dynamical systems under perturbations. On the other hand, topological stability—a property ensuring that a system remains conjugate to its small perturbations—has deep connections with structural stability and hyperbolicity. Investigating the interplay between these two concepts has led to significant insights into the long-term behavior of dynamical systems.
\\ \\
Parallel to these developments, the application of ideals and Furstenberg families has provided a powerful framework for refining classical dynamical notions, as demonstrated in \cite{attar, ri, shao, wang}. By incorporating combinatorial and set-theoretic techniques, these tools have enabled finer classifications of recurrence, transitivity, and other dynamical properties.
\\ \\
In this paper, we bridge these two lines of research by examining the interaction between shadowing and topological stability through the lens of ideals. Our main results, obtained in 2024, focus on topological stability within the setting of uniform transformation groups. Notably, we extend previous work by considering not only finitely generated phase groups but also more general phase groups, thereby broadening the scope of applicability.
\\ \\
Originally, our approach did not explicitly involve ideals; however, during the final stages of preparation, we encountered \cite{cec}, which motivated us to refine and enrich our framework using ideal-theoretic methods. Although our core proofs were developed independently, this perspective allowed us to unify and strengthen our results, offering new connections between shadowing and stability.
\\ \\
The paper is organized as follows. Section 2 provides the necessary background, reviewing key definitions and preliminary results. In Sections 3 through 5, we develop the fundamental definitions and present our main theorems concerning shadowing(modulo an ideal) properties. Section 6 introduces the concept of expansive modulo an ideal, accompanied by a counterexample demonstrating that this notion is not equivalent to classical expansive. Finally, Section 7 establishes our central result: for compact Hausdorff transformation semigroups, the conjunction of shadowing and expansivity modulo an ideal implies topological stability modulo the same ideal.
%%%%%%%%%%%%%%%%%%%%%%%%%%%%%%%%%%%%
\section{Preliminaries}
%%%%%%%%%%%%%%%%%%%%%%%%%%%%%%%%%%%%
\subsection{Background on transformation semigroups.}
By a (left topological) transformation semigroup $(T,X,\rho)$ or simply $(T,X)$, we mean a topological space $X$ (phase space), discrete topological semigroup $T$ (phase semigroup) with identity element $e$
and continuous map $\mathop{\rho:T\times X\to X}\limits_{\SP\SP(t,x)\mapsto tx}$ (action map) such that
\begin{itemize}
\item $ex=x$ for all $x\in X$,
\item $(ts)x=t(sx)$ for all $x\in X$, and $s,t\in T$.
\end{itemize}
If $T$ is a discrete topological group, then we call the transformation semigroup $(T,X)$, a transformation group.
\\
In transformation semigroup (resp. group) $(T,X,\rho)$ for each $t\in T$ the map $\rho^t:X\to X$
such that $\rho_t(x)=tx$ ($x\in X$) is continuous (resp. homeomorphism).
Let's denote the collection of all actions of semigroup $T$ on $X$ by $Act(T,X)$,
\\
For more details on transformation semigroups see~\cite{ellis, kawa}.
%%%%%%%%%%%%%%%%%%%%%%%%%%%%%%%%%%%%
\begin{note}\label{taha10}
Every dynamical system can be considered as a transformation semigroup, in the following way: one may consider continuous self--map
$f:X\to X$ as transformation semigroup $(\mathbb{N}\cup\{0\},X,\rho_f)$ where $\rho_f(n,x)=f^n(x)$ for all $x\in X$ and $n\geq0$, also
one may consider homeomorphism $g:X\to X$ as transformation group $(\mathbb{Z},X,\varrho_g)$ where $\varrho_g(n,x)=g^n(x)$
for all $x\in X$ and $n\in\mathbb{Z}$.
\\
So the collection of all continuous self--maps on topological space $X$ is in one to one correspondence
with $Act(\mathbb{N}\cup\{0\},X)$.
Also the collection of all self--homeomorphisms on topological space $X$ is
in one to one correspondence with $Act(\mathbb{Z},X)$.
\end{note}
%%%%%%%%%%%%%%%%%%%%%%%%%%%%%%%%%%%%
\subsection{Background on uniform spaces}
Suppose $\mathcal K$ is a collection of subsets of $X\times X$ such that (suppose $\Delta_X=\{
(x,x):x\in X\}$):
\begin{itemize}
\item $\forall\alpha\in\mathcal{K}\:(\Delta_X\subseteq\alpha)$,
\item $\forall\alpha,\beta\in\mathcal{K}\:(\alpha\cap\beta\in\mathcal{K})$,
\item $\forall\alpha\in\mathcal{K}\:\forall\beta\subseteq X\times X\:(\alpha\subseteq\beta\Rightarrow\beta\in\mathcal{K})$,
\item $\forall\alpha\in\mathcal{K}\:(\alpha^{-1}\in\mathcal{K})$,
\item $\forall\alpha\in\mathcal{K}\:\exists\beta\in\mathcal{K}\:(\beta\circ\beta\subseteq\alpha)$,
\end{itemize}
where for $\lambda,\mu\subseteq X\times X$ we have $\lambda^{-1}=\{(y,x):(x,y)\in\lambda\}$
and $\lambda\circ\mu=\{(x,z):\exists y\:((x,y)\in\mu\wedge(y,z)\in\lambda)\}$, then we call $\mathcal{K}$ a uniform structure on $X$.
\\
If $\mathcal{K}$ is a uniform structure on $X$, we call $(X,\mathcal{K})$ a uniform space. In uniform space $(X,\mathcal{K})$
for each $\alpha\in\mathcal{K}$ and $x\in X$ let $\alpha[x]:=\{y:(x,y)\in\alpha\}$, then
$\tau_{\mathcal K}:=\{U\subseteq X:\forall x\in U\:\exists\alpha\in\mathcal{K}\:(\alpha[x]\subseteq U)\}$ is uniform topology on $X$
(arised from uniform structure $\mathcal K$). Equip uniform space $(X,\mathcal{K})$ with uniform topology $\tau_{\mathcal K}$.
\\
In uniform spaces $(X,\mathcal{K}),(Y,\mathcal{H})$ we call $f:(X,\mathcal{K})\to(Y,\mathcal{H})$ uniform
continuous if for each $\alpha\in\mathcal{H}$ there exists $\beta\in\mathcal{K}$ such that
$f(\beta):=\{(f(x),f(y)):(x,y)\in\beta\}\subseteq\alpha$.
\\
The topological space $Z$ is called uniformizable with compatible uniform structure $\mathcal D$, if
$\mathcal D$ is a uniform structure on $Z$ and uniform topology $\tau_\mathcal{D}$ coincides with the topology
of $Z$. Note that every compact uniform space has a unique compatible
uniform structure (see e.g.~\cite{gal}), moreover every
compact Hausdorff topological space $Z$ is uniformizable
with unique compatible uniform structure
$\{\alpha\subseteq Z\times Z:\Delta_Z\subseteq\alpha^\circ\}$.
Also every metric space $(Z,d)$ is uniformizable with compatible uniform structure
$\mathcal{K}_d:=\{\alpha\subseteq Z\times Z:\exists\varepsilon>0\:(\{(x,y)\in Z\times Z:d(x,y)<\varepsilon\}\subseteq\alpha)\}$.
\\
The net $\{x_i\}_{i\in I}$ in uniform space $(X,\mathcal{K})$
is a Cauchy net if for each $\alpha\in\mathcal{K}$ there exists $j\in I$ such that
for all $p,q\geq j$ we have $(x_p,x_q)\in\alpha$. The uniform space $(X,\mathcal{K})$
is complete if each Cauchy net in $X$ is a convergent net. Every compact uniform space is complete.
\\
For more details on uniform spaces see~\cite{dug, engel}.
%%%%%%%%%%%%%%%%%%%%%%%%%%%%%%%%%%%%
\subsection{Background and preliminaries on uniform convergence topology}
%%%%%%%%%%%%%%%%%%%%%%%%%%%%%%%%%%%%
In uniform space $(X,\mathcal{K})$, let's denote
\begin{itemize}
\item the collection of all self--maps on $X$ by $F(X)$,
\item the collection of all continuous self--maps on $X$ by $C(X)$,
\item the collection of all self--homeomorphisms on $X$ by $Homeo(X)$.
\end{itemize}
Also for each $\theta\in\mathcal{K}$ let:
\[B_\theta:=\{(f,g)\in F(X)\times F(X):\forall x\in X\:\:(f(x),g(x))\in\theta\}\:,\]
then
\[\mathcal{H}_{\mathcal K}:=\{\alpha\subseteq F(X)\times F(X):\exists\lambda\in\mathcal{K}\:\:B_\lambda\subseteq\alpha\}\]
is a uniform structure on $F(X)$.
For $M\subseteq F(X)$, consider $M$ with induced uniformity
$\mathcal{H}_{\mathcal K}^M:=\{\alpha\cap(M\times M):\alpha\in \mathcal{H}_{\mathcal K}\}$ which is compatible with subspace
topology on $M$ (inherited from $(F(X),\mathcal{H}_{\mathcal K})$).
\\
It is known that uniform convergence topology on $F(X)$ provided by $\mathcal{H}_{\mathcal K}$, i.e., a net
	$\{f_i\}_{i\in I}$ converges to $f$ in $(F(X),\mathcal{H}_{\mathcal K})$ if and only if
	$\{f_i\}_{i\in I}$ converges to $f$ uniformly on $(X,\mathcal{K})$. So one may call uniform topology induced by $\mathcal{H}_{\mathcal K}$ as uniform convergence topology. More details are in the following notes.
%%%%%%%%%%%%%%%%%%%%%%%%%%%%%%%%%%%%
\begin{note}\label{taha15}
In uniform space $(X,\mathcal{K})$, uniform convergence topology on $F(X)$ coincides with uniform
topology of $(F(X), \mathcal{H}_{\mathcal K})$.
\end{note}
%%%%%%%%%%%%%%%%%%%%%%%%%%%%%%%%%%%%
\begin{proof}
Consider $f\in F(X)$ and a net $\{f_i\}_{i\in I}$ in $F(X)$.
\\
First suppose $\{f_i\}_{i\in I}$ converges to $f$ in uniform space $(F(X),\mathcal{H}_\mathcal{K})$. For each
$\theta\in\mathcal{K}$, $f$ is an interior point of $B_\theta[f]$, thus there exists $j\in I$ such that
$f_i\in B_\theta[f]$ for each $i\geq j$, i.e. $(f,f_i)\in B_\theta$ for each $i\geq j$, thus $(f(x),f_i(x))\in \theta$
for each $i\geq j$ and $\theta\in\mathcal{K}$. Hence
\begin{equation}\label{uniform}
\forall\theta\in\mathcal{K}\:\exists j\in I\:\forall i\geq j\:\forall x\in X\: (f(x),f_i(x))\in \theta
\end{equation}
so $\{f_i\}_{i\in I}$ converges uniformly to $f$.
\\
Now conversely, suppose $\{f_i\}_{i\in I}$ converges uniformly to $f$. For each open neighbourhood
$U$ of $f$ there exists $\alpha\in\mathcal{H}_\mathcal{K}$
such that $\alpha[f]\subseteq U$. There exists $\lambda\in\mathcal{K}$ such that $B_\lambda\subseteq\alpha$. Since $\{f_i\}_{i\in I}$ converges uniformly to $f$, there exists $s\in I$ such that $(f(x),f_i(x))\in \lambda$
for all $i\geq s,x\in X$, so $f_i\in B_\lambda[f]\subseteq \alpha[f]\subseteq U$ for each $i\geq s$. Therefore
$(f_i)_{i\in I}$ converges to $f$ in uniform space $(F(X),\mathcal{H}_\mathcal{K})$.
\end{proof}
%%%%%%%%%%%%%%%%%%%%%%%%%%%%%%%%%%%%
%%%%%%%%%%%%%%%%%%%%%%%%%%%%%%%%%%%%
%%%%%%%%%%%%%%%%%%%%%%%%%%%%%%%%%%%%
\begin{note}\label{taha20}
Consider uniform space $(X,\mathcal{K})$, self--map $g:X\to X$ and $L_g,R_g:F(X)\to F(X)$ with $L_g(f)=gf, R_g(f)=fg$ ($f\in F(X)$),
also equip $F(X)$ by uniform convergence topology. The following statements are valid:
\\
a. $C(X)$ is a closed subset of $(F(X),\mathcal{H}_{\mathcal K})$,
\\
b. the map $R_g:F(X)\to F(X)$ is continuous,
\\
c. if $g:X\to X$ is uniform continuous, then $L_g\restriction_{C(X)}:C(X)\to C(X)$ is continuous.
\end{note}
%%%%%%%%%%%%%%%%%%%%%%%%%%%%%%%%%%%%
\begin{proof}
{\bf a)} Suppose $f\in\overline{C(X)}$, then by Note~\ref{taha15} there exists a net $\{f_i\}_{i\in I}$ in $C(X)$
which converges uniformly to $X$.
Uniform limit of any continuous functions is continuous, thus $f\in C(X)$ and
$C(X)$ is a closed subset of $F(X)$.
\\
{\bf b)} If $\{f_i\}_{i\in I}$ is a net in $F(X)$ converges to $f$, then by Note~\ref{taha15}, it converges
uniformly to $f$, and we have~\ref{uniform}, thus
\[\forall\theta\in\mathcal{K}\:\exists j\in I\:\forall i\geq j\:\forall x\in X\: (f(g(x)),f_i(g(x)))\in \theta\]
i.e. $\{f_ig\}_{i\in I}$ converges to $fg$ uniformly. Use again Note~\ref{taha15} to complete the proof.
\end{proof}
%%%%%%%%%%%%%%%%%%%%%%%%%%%%%%%%%%%%
\begin{note}\label{taha30}
In Note~\ref{taha20} if $(X,\mathcal{H})$ is compact, then:
\\
a. $(C(X),\mathcal{H}_{\mathcal K}^{C(X)})$
	is a topological semigroup,
\\
b. $(Homeo(X),\mathcal{H}_{\mathcal K}^{Homeo(X)})$ is a topological group,
\\
c. if $X$ is Hausdorff too, then $(F(X),\mathcal{H}_{\mathcal K})$ and
	$(C(X),\mathcal{H}_{\mathcal K}^{C(X)})$ are complete.
\end{note}
%%%%%%%%%%%%%%%%%%%%%%%%%%%%%%%%%%%%
\begin{proof}
{\bf a)} Suppose $\{(f_i,g_i)\}_{i\in I}$ is a net in $C(X)\times C(X)$ converges to
$(f,g)$. Thus $\{f_i\}_{i\in I}$ converges uniformly to $f$ and
$\{g_i\}_{i\in I}$ converges uniformly to $g$.
Consider $\alpha\in \mathcal{H}_{\mathcal K}^{C(X)}$ there exists
$\beta^{-1}=\beta\in \mathcal{H}_{\mathcal K}^{C(X)}$ such that
$\beta\circ\beta\subseteq\alpha$. There exists $j_1\in I$ such that
\[\forall i\geq j_1\:\forall x\in X\: (f_i(x),f(x))\in\beta\:,\]
therefore
\begin{equation}\label{eq50}
\forall i\geq j_1\:\forall x\in X\: (f_i(g_i(x)),f(g_i(x)))\in\beta\:.
\end{equation}
The uniform space
$X$ is compact and $f:X\to X$ is continuous, thus $f:X\to X$
is uniformly continuous therefore by Note~\ref{taha20} the net
$\{fg_i\}_{i\in I}$ converges uniformly to $fg$. There exists $j_2\in I$ such that:
\begin{equation}\label{eq55}
\forall i\geq j_2\:\forall x\in X\: (f(g_i(x)),f(g(x)))\in\beta\:.
\end{equation}
$I$ is a directed set, there exists $j\in I$ such that $j\geq j_1$ and $j\geq j_2$,
by~\ref{eq50} and~\ref{eq55}
\[\forall i\geq j\: \forall x\in X\:(f_i(g_i(x)),f(g(x)))\in\beta\circ\beta\subseteq\alpha\:\]
Hence $\{f_ig_i\}_{i\in I}$ converges uniformly to $fg$, thus
$\mathop{C(X)\times C(X)\to C(X)}\limits_{\SP\SP\SP(h,k)\mapsto hk}$ is continuous and
$C(X)$ is a topological semigroup.
\\
{\bf b)} Suppose $\{f_i\}_{i\in I}$ is a net in $Homeo(X)$ converges uniformly to
$f\in Homeo(X)$. Consider $\alpha\in\mathcal{K}$, since $f^{-1}:X\to X$ is continuous and
$X$ is a compact uniform space, $f^{-1}:X\to X$ is uniformly continuous.
Hence there exists
$\beta\in\mathcal{K}$ such that $\{(f^{-1}(x),f^{-1}(y)):(x,y)\in\beta\}\subseteq\alpha$.
There exists $j\in I$ such that
\[\forall i\geq j\:\forall x\in X\:(f_i(x),f(x))\in\beta\]
thus
\[\forall i\geq j\:\forall x\in X\:(x,f(f_i^{-1}(x)))=(f_i(f_i^{-1}(x)),f(f_i^{-1}(x)))\in\beta\]
hence
\[\forall i\geq j\:\forall x\in X\:((f^{-1}(x),f_i^{-1}(x))=(f^{-1}(x),f^{-1}(f(f_i^{-1}(x))))\in\alpha\:.\]
thus $\{f_i^{-1}\}_{i\in I}$ converges uniformly to $f^{-1}$. Therefore
$\mathop{Homeo(X)\to Homeo(X)}\limits_{\SP h\mapsto h^{-1}}$ is continuous, use (a)
to complete the proof of (b).
\\
{\bf c)} Suppose $X$ is compact Hausdorff, thus for each $\beta\in\mathcal{K}$ there exists
$\alpha\in\mathcal{K}$ such that $\overline{\alpha}=\alpha\subseteq\beta$.
If $\{f_i\}_{i\in I}$ is a Cauchy net in $F(X)$, then for each $x\in X$, $\{f_i(x)\}_{i\in I}$ is
a Cauchy net in $X$, and converges to a point of $X$ like $f(x)$. For each
$\theta\in\mathcal{H}_\mathcal{K}$ there exists $\overline{\alpha}=\alpha\in\mathcal{K}$ such that
$B_\alpha\subseteq\theta$. There exists $j\in I$ such that $(f_p,f_q)\in B_\alpha$ for each $p,q\geq j$.
Thus $(f_p(x),f_q(x))\in \alpha$ for each $p,q\geq j$ and $x\in X$. Hence
$(f_p(x),f(x))=\mathop{\lim}\limits_{q\in I}(f_p(x),f_q(x))\in\overline{\alpha}=\alpha$ for each
$q\geq j$ and $x\in X$. Hence $(f_p,f)\in B_\alpha\subseteq\theta$ for each $p\geq j$ which shows
convergent of $\{f_i\}_{i\in I}$ (to $f$).
\end{proof}
%%%%%%%%%%%%%%%%%%%%%%%%%%%%%%%%%%%%
\subsection{Other notations and conventions}
Suppose $\mathcal{I}$ is a collection of subsets of $M$, we say $\mathcal I$ is
an ideal on $M$ if
\begin{itemize}
\item $\mathcal{I}\neq\varnothing$,
\item $\forall A,B\in\mathcal{I}\:(A\cup B\in\mathcal{I})$,
\item $\forall A\in\mathcal{I}\:\forall B\subseteq A\:(B\in\mathcal{I})$.
\end{itemize}
\noindent $\mathcal{P}_{fin}(M)$ denotes the ideal of finite subsets of $M$, $\mathcal{P}(M)$
(power set of $M$) is the biggest ideal on $M$, and $\{\varnothing\}$ is the smallest ideal on $M$.
\\
In transformation semigroup $(T,X,\rho)$ (or simply $(T,X)$) with uniformizable phase space
$X$, sometimes chosen compatible uniform structure is important, so if we want to make
emphasis on compatible uniform structure $\mathcal K$ on $X$, we use notation
$(T,(X,\mathcal{K}),\rho)$ (or simply $(T,(X,\mathcal{K}))$) too. 
Whenever either $\mathcal K$ is clear from text or $X$ has just one compatible uniform structure like $\mathcal K$ (e.g., when $X$ is compact Hausdorff), instead of writing
$(T,(X,\mathcal{K}))$ we write briefly
$(T,X)$.
%%%%%%%%%%%%%%%%%%%%%%%%%%%%%%%%%%%%
\begin{convention}
In this text consider uniform transformation semigroup $(T,(X,\mathcal{K}),\rho)$
and ideal $\mathcal I$ on $T$. 
\end{convention}
%%%%%%%%%%%%%%%%%%%%%%%%%%%%%%%%%%%%
\section{Shadowing with respect to a set}
%%%%%%%%%%%%%%%%%%%%%%%%%%%%%%%%%%%%
\noindent In this section we pay attention to interaction between ``having shadowing
property with respect to $A$'' for different nonempty subsets $A$ of phase semigroup
and their interaction. In compact uniform transformation (semi)group
$(T,(Z,\mathcal{K}))$ for each nonempty subset $A$ of $Z$ and each $s$
belongs to sub(semi)group generated by $A$ we prove $(T,(Z,\mathcal{K}))$
has shadowing property with respect to $A$ if and only if it has shadowing property
with respect to $A\cup\{s\}$.
%%%%%%%%%%%%%%%%%%%%%%%%%%%%%%%%%%%%
\begin{definition}\label{salam110}
In uniform transformation semigroup $(T,(X,\mathcal{K}))$ for
a sequence $(x_t)_{t\in T}$ in
$X$, $\theta\in\mathcal{K}$, $x\in X$ and
nonempty subset $A$ of $T$ we say:
\begin{itemize}
\item $(x_t)_{t\in T}$ is a $\theta-$pseudo orbit (with respect to $A$) if:
\[\forall a\in A\:\forall t\in T\:((ax_t,x_{at})\in\theta)\:,\]
\item $x$ is a $\theta-$trace of $(x_t)_{t\in T}$ if:
\[\forall t\in T\:((tx,x_{t})\in\theta)\:.\]
\end{itemize}
For nonempty subset $A$ of $T$, we say the uniform transformation semigroup $(T,(X,\mathcal{K}))$ has shadowing property with resprct to $A$
if for each $\alpha\in\mathcal{K}$ there exists
$\beta\in\mathcal{K}$ such that every $\beta-$pseudo orbit with respect to $A$ has an $\alpha-$trace.
\end{definition}
%%%%%%%%%%%%%%%%%%%%%%%%%%%%%%%%%%%%
\begin{example}\label{salam120}
For each $\alpha\in\mathcal{K}$, every sequence $(x_t)_{t\in T}$ is an
$\alpha-$pseudo orbit with respect to $\{e\}$.
\\
Moreover, $(T,(X,\mathcal{K}))$ has shadowing property with respect to $T$. For each $\alpha\in\mathcal{K}$, if $(x_t)_{t\in T}$ is an
$\alpha-$pseudo orbit with respect to $T$, then $x_e$ is an $\alpha-$trace of $(x_t)_{t\in T}$.
\end{example}
%%%%%%%%%%%%%%%%%%%%%%%%%%%%%%%%%%%%
\begin{example}\label{salam140}
If $\varnothing\neq B\subseteq A\subseteq T$ and
$(T,(X,\mathcal{K}))$ has shadowing property with respect to $B$, then it has shadowing property with respect to $A$.
\end{example}
%%%%%%%%%%%%%%%%%%%%%%%%%%%%%%%%%%%%
\begin{corollary}\label{salam300}
The following statements are equivalent:
\\
a) there exists $A\in\mathcal{I}$ such that $(T,X)$ has shadowing property
	with respect to $A$,
\\
b) for each $C\in\mathcal{I}$
	there exists $B\in\mathcal{I}$ such that $C\subseteq B$ and
	$(T,X)$ has shadowing property
	with respect to $B$.
\end{corollary}
%%%%%%%%%%%%%%%%%%%%%%%%%%%%%%%%%%%%
\begin{proof} ``(a) $\Rightarrow$ (b)'' Suppose there exists $A\in\mathcal{I}$ 
such that $(T,X)$ has shadowing property with respect to $A$. Choose $C\in\mathcal{I}$,
by Example~\ref{salam140}, $(T,X)$ has shadowing property with respect to $B:=A\cup C$.
\end{proof}
%%%%%%%%%%%%%%%%%%%%%%%%%%%%%%%%%%%%
\begin{lemma}\label{salam150}
Suppose $\varnothing\neq A\subseteq T$.
$(T,(X,\mathcal{K}))$ has shadowing property with respect to $A$ if and only if it
has shadowing property with respect to $A\cup\{e\}$.
\end{lemma}
%%%%%%%%%%%%%%%%%%%%%%%%%%%%%%%%%%%%
\begin{proof}
Use the fact that for each $\beta\in\mathcal{K}$, a sequence $(x_t)_{t\in T}$ is a $\beta-$pseudo orbit with respect to $A$ if and only if
it is a $\beta-$pseudo orbit with respect to $A\cup\{e\}$.
\end{proof}
%%%%%%%%%%%%%%%%%%%%%%%%%%%%%%%%%%%%
\begin{lemma}\label{salam160}
Suppose $X$ is compact and $r,s\in A\subseteq T$.
Then $(T,(X,\mathcal{K}))$ has shadowing property with respect to $A$
if and only if it
has shadowing property with respect to $A\cup\{sr\}$.
\end{lemma}
%%%%%%%%%%%%%%%%%%%%%%%%%%%%%%%%%%%%
\begin{proof}
Suppose $(T,(X,\mathcal{K}))$ has shadowing property with respect to $A\cup\{sr\}$
and consider $\alpha\in\mathcal{K}$.
There exists $\beta\in\mathcal{K}$ such that each $\beta-$pseudo orbit with respect to $A\cup\{sr\}$ has an $\alpha-$trace.
Choose $\psi\in\mathcal{K}$ such that $\psi\circ\psi\subseteq\beta$. Since $s$ is a continuous self--map on compact uniform space
$(X,\mathcal{K})$, it is uniformly continuous too. So there exists $\lambda\in\mathcal{K}$ such that $\{(sx,sy):(x,y)\in\lambda\}\subseteq\psi$.
Let $\theta=\lambda\cap\psi$, then $\theta\in\mathcal{K}$.
\\
Suppose $(x_t)_{t\in T}$ is a $\theta-$pseudo orbit with respect to $A$, then
for each $a\in A$ and $t\in T$,
$(ax_t,x_{at})\in\theta\subseteq\beta$. In particular for each
$t\in T$, $(sx_{rt},x_{srt})\in\theta\subseteq\psi$ and
$(rx_t,x_{rt})\in\theta\subseteq\lambda$ hence $(srx_t,sx_{rt})\in\psi$
so $(srx_t,x_{srt})\in\psi\circ\psi\subseteq\beta$.
Therefore $(x_t)_{t\in T}$ is a $\beta-$pseudo orbit with respect to $A\cup\{sr\}$
and has an $\alpha-$trace.
\\
Use Example~\ref{salam140} to complete the proof.
\end{proof}
%%%%%%%%%%%%%%%%%%%%%%%%%%%%%%%%%%%%
\begin{lemma}\label{salam170}
Suppose $X$ is compact, $T$ is group and $s\in A\subseteq T$.
Then $(T,(X,\mathcal{K}))$ has shadowing property with respect to $A$
if and only if it
has shadowing property with respect to $A\cup\{s^{-1}\}$.
\end{lemma}
%%%%%%%%%%%%%%%%%%%%%%%%%%%%%%%%%%%%
\begin{proof}
Suppose $(T,(X,\mathcal{K}))$ has shadowing property with respect to $A\cup\{s^{-1}\}$
and consider $\alpha\in\mathcal{K}$.
There exists $\beta\in\mathcal{K}$ such that each $\beta-$pseudo orbit with respect to $A\cup\{s^{-1}\}$ has an $\alpha-$trace.
Since $s^{-1}$ is a continuous self--map on compact uniform space
$(X,\mathcal{K})$, it is uniformly continuous too. So there exists $\lambda\in\mathcal{K}$ such that
$\{(s^{-1}x,s^{-1}y):(x,y)\in\lambda\}\subseteq\beta^{-1}\in\mathcal{K}$.
Let $\theta=\lambda\cap\beta$, then $\theta\in\mathcal{K}$.
\\
Suppose $(x_t)_{t\in T}$ is a $\theta-$pseudo orbit with respect to $A$, then
for each $a\in A$ and $t\in T$, $(ax_t,x_{at})\in\theta\subseteq\beta$. In particular for each $t\in T$,
$(sx_{t},x_{st})\in\theta\subseteq\lambda$ thus
$(x_t,s^{-1}x_{st})=(s^{-1}sx_{t},s^{-1}x_{st})\in\beta^{-1}$. Therefore
$(s^{-1}x_{st},x_t)\in\beta$ for each $t\in T$, hence $(s^{-1}x_{ss^{-1}t},x_{s^{-1}t})=(s^{-1}x_t,x_{s^{-1}t})\in\beta$ for all $t\in T$.
Thus $(x_t)_{t\in T}$ is a $\beta-$pseudo orbit with respect to $A\cup\{s^{-1}\}$
and has an $\alpha-$trace.
\\
Use Example~\ref{salam140} to complete the proof.
\end{proof}
%%%%%%%%%%%%%%%%%%%%%%%%%%%%%%%%%%%%
\begin{theorem}\label{salam180}
Suppose $X$ is compact $\varnothing\neq A\subseteq T$ and $s$ belongs to subsemigroup
generated by $A$. Then $(T,(X,\mathcal{K}))$ has shadowing property with respect to $A$ if and only if it
has shadowing property with respect to $A\cup\{s\}$.
\\
Moreover if $T$ is a group too, then we may suppose $s$
belongs to subgroup
generated by $A$.
\end{theorem}
%%%%%%%%%%%%%%%%%%%%%%%%%%%%%%%%%%%%
\begin{proof}
Use 
Example~\ref{salam140}, Lemma~\ref{salam160}, Lemma~\ref{salam170} and
induction.
\end{proof}
%%%%%%%%%%%%%%%%%%%%%%%%%%%%%%%%%%%%
%%%%%%%%%%%%%%%%%%%%%%%%%%%%%%%%%%%%
\begin{corollary}\label{salam310}
Suppose $X$ is compact $A$ is a finite subset of $T$ and $s$ is a generator of $T$.
Then $(T,(X,\mathcal{K}))$ has shadowing property with respect to $A$ if and only if it
has shadowing property with respect to $\{s\}$.
\end{corollary}
%%%%%%%%%%%%%%%%%%%%%%%%%%%%%%%%%%%%
\begin{proof}
First suppose $(T,(X,\mathcal{K}))$ has shadowing property with respect to $A$, then 
$(T,(X,\mathcal{K}))$ has shadowing property with respect to $A\cup\{s\}$ by
Example~\ref{salam140}. By Theorem~\ref{salam180},
$(T,(X,\mathcal{K}))$ has shadowing property with respect to $\{s\}$. Use Example~\ref{salam140} to complete the proof. 
\end{proof}
%%%%%%%%%%%%%%%%%%%%%%%%%%%%%%%%%%%%
\section{Shadowing with respect to $\{e\}$}
\noindent In this section we prove that
uniform transformation semigroup $(T,(X,\mathcal{K}))$ has shadowing property
with respect to $\{e\}$ if and only if $T=\{e\}$ or $X$ has trivial topology (i.e. $\{X,\varnothing\}$).
%%%%%%%%%%%%%%%%%%%%%%%%%%%%%%%%%%%%
%%%%%%%%%%%%%%%%%%%%%%%%%%%%%%%%%%%%
\begin{lemma}\label{salam199}
If $(T,(X,\mathcal{K}))$ has shadowing property
with respect to $\{e\}$, then
for all opene subsets $U,V$ of $X$ and $t\in T\setminus\{e\}$, $tU\cap V\neq\varnothing$.
\end{lemma}
%%%%%%%%%%%%%%%%%%%%%%%%%%%%%%%%%%%%
\begin{proof}
Suppose $U,V$ are opene subsets of $X$ and $T\neq\{e\}$. Choose $x\in U,y\in V$. There exist $\alpha_1=\alpha_1^{-1},\alpha_2=\alpha_2^{-1}\in\mathcal{K}$ such that
$\alpha_1[x]\subseteq U,\alpha_2[y]\subseteq V$. Let $\theta=\alpha_1\cap\alpha_2$.
There exists $\psi\in\mathcal{K}$ such that every $\psi-$pseudo orbit with respect to $\{e\}$ has a $\theta-$trace.
Let $z_e=x$ and $z_t=y$ for $t\neq e$. Then $(z_t)_{t\in T}$ is a $\psi-$pseudo orbit with respect to $\{e\}$ and has a
$\theta-$trace like $z$. We have $z\in\alpha_1[x](\subseteq U)$, since $(z,x)=(ez,z_e)\in\theta\subseteq\alpha_1$.
Choose $t\in T\setminus\{e\}$, then
$(tz,y)=(tz,z_t)\in\theta\subseteq\alpha_2$
so $tz\in\alpha_2[y]\subseteq V$. Therefore $tz\in tU\cap V$, in particular $\varnothing\neq tU\cap V\subseteq TU\cap V$.
\end{proof}
%%%%%%%%%%%%%%%%%%%%%%%%%%%%%%%%%%%%
\begin{lemma}\label{salam190}
If $(T,(X,\mathcal{K}),\rho)$ has shadowing property
with respect to $\{e\}$ and $T\neq\{e\}$, then $\overline{U}=X$ for all opene subset $U$ of $X$.
\end{lemma}
%%%%%%%%%%%%%%%%%%%%%%%%%%%%%%%%%%%%
\begin{proof}
Suppose $U$ is an open subset of $X$, for
each $x\in U$, $t\in T\setminus\{e\}$ and open neighbourhood $W$ of $tx$ there exists
open subset $V$ of $x$ such that $tV\subseteq W$ (since $\rho^t:X\to X$ is continuous).
By Lemma~\ref{salam199}, $(W\cap U\supseteq)tV\cap U\neq\varnothing$. Since all
open neighbourhood of $tx$ like $W$ has nonempty intersection with $U$, we have
$tx\in\overline{U}$. Therefore $(T\setminus\{e\})x\subseteq \overline{U}$. Hence
$Tx=(T\setminus\{e\})x\cup\{ex\}=(T\setminus\{e\})x\cup\{x\}\subseteq \overline{U} \cup U=\overline{U}$. However $x\in U$ is arbitrary. thus
$TU=\bigcup\{Tz:z\in U\}\subseteq \overline{U}$.
\\
Choose $t\in T\setminus\{e\}$.
If $U$ is an opene subset of $X$ such that $\overline{U}\neq X$, thus
$V=X\setminus\overline{U}$ is an opene subset of $X$ too. By
Lemma~\ref{salam199}~(2), $tU\cap V\neq\varnothing$. By the first part of proof,
$tU\subseteq TU\subseteq \overline{U}$, hence $\overline{U}\cap V\neq\varnothing$
which is a contradiction and leads to the desired result.
\end{proof}
%%%%%%%%%%%%%%%%%%%%%%%%%%%%%%%%%%%%
\begin{remark}\label{salam195}
Let's recall that a topological space is completely regular if for each closed subset $C$
of $X$ and $x\in X\setminus C$ there exists continuous map $f:X\to[0,1]$ such
that $f(x)=1$ and $f(C)=\{0\}$. Moreover, a topological space $X$ is regular if
for each closed subset $C$
of $X$ and $x\in X\setminus C$ there exist disjoint open sets $U,V$ such that
$C\subseteq U, x\in V$. So each completely regular space is regular.
On the other hand each uniform topological space is completely
regular~\cite[Theorem 38.2]{will}.
\end{remark}
%%%%%%%%%%%%%%%%%%%%%%%%%%%%%%%%%%%%
\begin{theorem}
Uniform transformation semigroup $(T,(X,\mathcal{K}))$ has shadowing property
with respect to $\{e\}$ if and only if either $T=\{e\}$ or $X$ is trivial topological space.
\end{theorem}
%%%%%%%%%%%%%%%%%%%%%%%%%%%%%%%%%%%%
\begin{proof}
Suppose $(T,(X,\mathcal{K}))$ has shadowing property
with respect to $\{e\}$ and $U$ is an opene subset of $X$. Choose $x\in U$
by Remark~\ref{salam195} there exist open sets like $V,W$ such that
$x\in V,X\setminus U\subseteq W$ and $V\cap W=\varnothing$, so by Lemma~\ref{salam190},
$X=\overline{V}\subseteq X\setminus W\subseteq U$ which shows $U=X$. Thus
$X$ has trivial topology.
\end{proof}
%%%%%%%%%%%%%%%%%%%%%%%%%%%%%%%%%%%%
\section{$\mathcal{I}-$shadowing property}
%%%%%%%%%%%%%%%%%%%%%%%%%%%%%%%%%%%%
\begin{definition}
We say the uniform transformation semigroup $(T,(X,\mathcal{K}))$ has shadowing property modulo ideal $\mathcal I$ (or briefly
${\mathcal I}-$shadowing property)
if there exists nonempty $A\in\mathcal{I}$
such that $(T,(X,\mathcal{K}))$ has shadowing property with respect to $A$
(see Corollary~\ref{salam300}).
\\
We say the uniform transformation semigroup $(T,(X,\mathcal{K}))$ has shadowing property
or pseudo orbit tracing property (POTP) if it has $\mathcal{P}_{fin}(T)-$shadowing property.
\end{definition}
\noindent Let's mention that we call a transformation semigroup $(T,X)$ effective if
for each distinct $s,t\in T$ there exists $x\in X$ such that $tx\neq sx$.
%%%%%%%%%%%%%%%%%%%%%%%%%%%%%%%%%%%%
\begin{example}\label{salam130}
In finite uniform effective transformation semigroup $(T,(X,\mathcal{K}))$, suppose $V=\bigcap\mathcal{K}$, then
$\mathcal{K}=\{D\subseteq X\times X:V\subseteq D\}$. So $T$ is finite.
Suppose $(x_t)_{t\in T}$ is a $V-$pseudo orbit with respect to $T$, then for all $a\in T$ we have
$(ax_e,x_a)\in V$, in particular $x_e$ is an $\alpha-$trace of $(x_t)_{t\in T}$ for all $t\in T$.
Therefore $(T,(X,\mathcal{K}))$ has shadowing property.
\end{example}
%%%%%%%%%%%%%%%%%%%%%%%%%%%%%%%%%%%%
%%%%%%%%%%%%%%%%%%%%%%%%%%%%%%%%%%%%
%%%%%%%%%%%%%%%%%%%%%%%%%%%%%%%%%%%%
\subsection{Compatibility with some of the other generalizations of shadowing in dynamical systems}\label{subsection}
Let's bring here 5 witnesses:
\\
%%%%%%%%%%%%%%%%%%%%%%%%%%%%%%%%%%%%
$\bullet$ in compact metric space $(X,d)$, homeomorphismm $f:X\to X$ has POTP
in the sense of~\cite[Definition 3]{walters} if and only if $(\mathbb{Z},X,\varrho_f)$ has POTP
in the sense of Definition~\ref{salam110},
\\
$\bullet$ in metric space $(X,d)$, homeomorphismm $f:X\to X$ has POTP
in the sense of \cite[Definition 1.3]{pilyugin} if and only if $(\mathbb{Z},(X,\mathcal{K}_d),\varrho_f)$ has POTP
according to Definition~\ref{salam110},
\\
$\bullet$ in metric space $(X,d)$, continuous (non--bijective) map $f:X\to X$ has POTP
in the sense of \cite[Section 2.3]{aoki} if and only if $(\mathbb{N}\cup\{0\},(X,\mathcal{K}_d),\rho_f)$ has POTP
according to Definition~\ref{salam110},
\\
$\bullet$ in compact uniform space $X$, homeomorphism $f:X\to X$ has POTP
in the sense of \cite[Definition 2.2]{yan} if and only if $(\mathbb{Z},X,\varrho_f)$ has POTP
according to Definition~\ref{salam110},
%\\
%$\bullet$ in uniform space $(X,\mathcal{K})$, uniformly continuous $f:(X,\mathcal{K})\to (X,\mathcal{K})$ has POTP
%in the sense of~\cite{ahmadi} \footnote{authors of \cite{ahmadi} called this property topological shadowing, however their definition is different with topological shadowing~\cite{top}, they indicate this note for just metric spaces, but it is valid for their own definition too} if and only if $(\mathbb{Z},(X,\mathcal{K}),\rho_f)$ has POTP
%according to Definition~\ref{salam110}.
\\
\subsection{Compatibility with some of the other generalizations of shadowing in transformation groups} In transformation group
$(T,X,\rho)$:
\\
$\bullet$ suppose $T$ is finitely generated and $X$ is a compact metric space. So $(T,X)$ has POTP in the sense of~\cite[Definition 2.6]{lee} if and only if $(T,X)$ has POTP
according to Definition~\ref{salam110},
\\
$\bullet$ suppose $T$ is finitely generated, $(X,d)$ is a metric space and for each $t\in T$ the map
$\rho^t:(X,\mathcal{K})\to(X,\mathcal{K})$ is equicontinuous. $(T,(X,d))$ has POTP in the sense of~\cite[Definition 2.2]{osipov} if and only if $(T,(X,\mathcal{K}_d))$ has POTP
according to Definition~\ref{salam110}.
\\
$\bullet$ suppose $T$ is finitely generated and $(X,\mathcal{K})$ is a uniform space.
$(T,(X,\mathcal{K}))$ has POTP (shadowing property) in the sense
of~\cite[Definition 2.4]{stab} if and only if $(T,(X,\mathcal{K}))$ has POTP
according to Definition~\ref{salam110}.
%%%%%%%%%%%%%%%%%%%%%%%%%%%%%%%%%%%%
\subsection{A counterexample}
By the following example shadowing and $\mathcal{I}-$shadowing are not equivalent.
%%%%%%%%%%%%%%%%%%%%%%%%%%%%%%%%%%%%
\begin{counterexample}
Let $X:=\left\{\dfrac{\pm1}n:n\geq1\right\}\cup\{0\}$ with induced metric $d$
of Euclidean line $\mathbb R$
and define $f:X\to X$ with:
\[f(x)=\left\{\begin{array}{lc}
\frac{1}{n-2} & x=1/n \: and \: n\in2\mathbb{Z}+1\:, \\
-x & otherwise\:.
\end{array}\right.\]
Also consider ideal $\mathcal{I}:=\mathcal{P}(2\mathbb{Z}+1)$ on $\mathbb{Z}$,
then transformation group $(\mathbb{Z},X,\varrho_f)$,
does not have shadowing property, although it has $\mathcal{I}-$shadowing property.
\end{counterexample}
%%%%%%%%%%%%%%%%%%%%%%%%%%%%%%%%%%%%
\begin{proof}
For $r>0$ let $\alpha_r:=\{(x,y)\in X\times X:|x-y|<r\}$.
By Corollary~\ref{salam310}, $(\mathbb{Z},X,\varrho_f)$ 
has shadowing property if and only if it has shadowing property 
with respect to $\{1\}$. Suppose $(\mathbb{Z},X,\varrho_f)$ 
has shadowing property with respect to $\{1\}$. Thus there exists $\delta>0$
such that each $\alpha_\delta-$pseudo orbit with respect to $\{1\}$ has an $\alpha_{1/8}-$trace.
Choose $N\geq2$ with $\frac{1}{2N+1}<\delta/2$, and let
\[x_n:=\left\{\begin{array}{lc}
f^{n}(\frac{-1}{2N+1}) & n\leq -1\:, \\ f^{n}(\frac{1}{2N+1}) & n\geq0\:.
\end{array}\right.\]
i.e., $(x_n)_{n\in\mathbb{Z}}=
(\cdots,f^{-2}(\frac{-1}{2N+1}),f^{-1}(\frac{-1}{2N+1}),\frac{1}{2N+1},
f(\frac{1}{2N+1}),f^2(\frac{1}{2N+1}),\cdots)$, then $(x_n)_{n\in\mathbb{Z}}$ is an
$\alpha_\delta-$pseudo orbit with respect to $\{1\}$, hence it has an $\alpha_{1/8}-$trace like $z$.
For each $n\in\mathbb{Z}$, $|f^n(z)-x_n|<1/8$, thus
$|1-f^N(z)|=|x_N-f^N(z)|<1/8$ thus $f^N(z)=1$
and $z=\frac{1}{2N+1}$, also $|-1-f^{-N}(z)|=|x_{-N}-f^{-N}(z)|<1/8$ thus $f^{-N}(z)=-1$
and $z=\frac{-1}{2N+1}$ which is a contradiction, thus 
$(\mathbb{Z},X,\varrho_f)$ does not have shadowing property.
\\
In order to show $(\mathbb{Z},X,\varrho_f)$ 
has $\mathcal{I}-$shadowing property, we prove $(\mathbb{Z},X,\varrho_f)$ 
has shadowing property with respect to $2\mathbb{Z}+1$. For $\varepsilon>0$ 
choose $N\geq3$ such that $1/N<\varepsilon$ let $\delta:=\frac{1}{2N+2}-\frac{1}{2N+4}(\leq1/8)$.
Suppose $(y_n)_{n\in\mathbb{Z}}$ is an $\alpha_\delta-$pseudo orbit 
with respect to $2\mathbb{Z}+1$. We show $(y_n)_{n\in\mathbb{Z}}$ has an $\alpha_\varepsilon-$trace
via the following 3 claims.
\\
{\bf Claim I.} If 
$\{y_n:n\in\mathbb{Z}\}\cap\left\{\dfrac{1}{2n+1}:n\in\mathbb{Z}\right\}\neq\varnothing$, 
then $(y_n)_{n\in\mathbb{Z}}$ is an orbit, i.e. 
there exists $w\in X$ such that $y_n=f^n(w)$ for all 
$n\in\mathbb{Z}$, in particular $w$ is an $\alpha_r-$trace 
of $(y_n)_{n\in\mathbb{Z}}$ for each $r>0$.
\\
{\it Proof of Claim I.} If there exists $p,q\in\mathbb{Z}$ such that $y_p=\frac{1}{2q+1}$, 
then 
\\
$q\in2\mathbb{Z}+1\vee q+1\in2\mathbb{Z}+1  $
\begin{eqnarray*}
& \Rightarrow & |f^q(y_p)-y_{p+q}|<1/8\vee|f^{q+1}(y_p)-y_{p+q+1}|<1/8 \\
& \Rightarrow & |1-y_{p+q}|<1/8\vee|-1-y_{p+q+1}|<1/8 \\
& \Rightarrow & y_{p+q}=1\vee y_{p+q+1}=-1 \\
& \Rightarrow & y_{p+q}=1\vee |1-y_{p+q}|=|f^{-1}(y_{p+q+1})-y_{p+q+1-1}|<1/8 \\
& \Rightarrow & y_{p+q}=1
\end{eqnarray*}
Thus for each $n\in\mathbb{Z}$ we have
\[|f^{2n+1}(y_{p+q-(2n+1)})-1|=|f^{2n+1}(y_{p+q-(2n+1)})-y_{p+q}|<1/8\:,\]
therefore $f^{2n+1}(y_{p+q-(2n+1)})=1$ and $y_{p+q-(2n+1)}=f^{-2n-1}(1)$.
In particular $y_{p+q+1}=f(1)=-1$. 
Also 
\[|f^{2n+1}(y_{p+q-(2n)})+1|=|f^{2n+1}(y_{p+q-(2n)})-y_{p+q+1}|<1/8\:,\]
therefore $f^{2n+1}(y_{p+q-(2n)})=-1$ and $y_{p+q-(2n)}=f^{-2n-1}(-1)=f^{-2n}(1)$. Hence:
\[\forall j\in\mathbb{Z}\: (y_j=f^j(f^{-p-q}(1))=f^j(\frac{1}{1+2p+2q}))\]
and $\frac{1}{1+2p+2q}$ is an $\alpha_r-$trace of $(y_n)_{n\in\mathbb{Z}}$ for each $r>0$.
\\
{\bf Claim II.} If 
$\{y_n:n\in\mathbb{Z}\}\subseteq\left\{\dfrac{\pm1}{2n}:n\geq N\right\}\cup\{0\}$, 
then $0$ is an $\alpha_\varepsilon-$trace of  $(y_n)_{n\in\mathbb{Z}}$.
\\
{\it Proof of Claim II.}  If $\{y_n:n\in\mathbb{Z}\}\subseteq\left\{\dfrac{\pm1}{2n}:n\geq N\right\}\cup\{0\}$,
then for each $n\in\mathbb{Z}$, $|f^n(0)-y_n|=|y_n|\leq \frac{1}{2N}<\varepsilon$.
\\
{\bf Claim III.} If 
$\{y_n:n\in\mathbb{Z}\}\subseteq\left\{\dfrac{\pm1}{2n}:n\geq 1\right\}\cup\{0\}$, and
$\{y_n:n\in\mathbb{Z}\}\cap\left\{\dfrac{\pm1}{2n}:1\leq n<N\right\}\neq\varnothing$
then  $(y_n)_{n\in\mathbb{Z}}$ is an orbit, i.e. 
there exists $w\in X$ such that $y_n=f^n(w)$ for all 
$n\in\mathbb{Z}$, in particular $w$ is an $\alpha_r-$trace 
of $(y_n)_{n\in\mathbb{Z}}$ for each $r>0$.
\\
{\it Proof of Claim III.}  
Suppose $\{y_n:n\in\mathbb{Z}\}\subseteq\left\{\dfrac{\pm1}{2n}:n\geq 1\right\}\cup\{0\}$
and there exists $p,q\in\mathbb{Z}$ with $1\leq |q|<N$ and $y_p=\frac{1}{2q}$.
Then $\{x\in X: \exists n\in\mathbb{Z}\: |f^m(y_p)-x|<\delta\}=\{\frac{\pm1}{2q}\}$. 
On the other hand for each $n\in\mathbb{Z}$, 
$\{|\frac{\pm1}{2q}-y_{p+2n+1}|\}\ni|f^{2n+1}(y_p)-y_{p+2n+1}|<\delta$
thus $y_{p+2n+1}=f^{2n+1}(y_p)=f(y_p)$. In particular $y_{p+1}=f^(y_p)=-y_p=\frac{-1}{2q}$. 
Similarly $y_{p+1+2n+1}=f^{2n+1}(y_{p+1})=f^{2n+2}(y_p)=y_p$ for each $n\in\mathbb{Z}$.
Hence $y_j=f^j(y_p)$ for each $j\in\mathbb{Z}$.
\\
Using the above Claims leads to the desired result.
\end{proof}
%%%%%%%%%%%%%%%%%%%%%%%%%%%%%%%%%%%%
\section{$\mathcal{I}-$expansiveness}
%%%%%%%%%%%%%%%%%%%%%%%%%%%%%%%%%%%%
\begin{definition}\label{salam210}
We say $(T,(X,\mathcal{K}))$ is expansive 
modulo ideal $\mathcal I$
(or briefly,  ${\mathcal I}-$expansive)
if there exists $\alpha\in\mathcal{K}$
(${\mathcal I}-$expansive index)
such that for all distinct $x,y\in X$ and $E\in\mathcal{I}$
there exists $t\in T\setminus E$ such that $(tx,ty)\notin\alpha$.
\end{definition}
%%%%%%%%%%%%%%%%%%%%%%%%%%%%%%%%%%%%
\begin{note}\label{salam215}
As a matter of fact $(T,(X,\mathcal{K}))$ is ${\mathcal I}-$expansive
with ${\mathcal I}-$expansive index $\psi$ if $\{t\in T:(tx,ty)\notin \psi\}\notin\mathcal{I}$ for all distinct $x,y\in X$.
Moreover if $\lambda\subseteq\psi$ and $\lambda\in\mathcal{K}$, then $\lambda$
is an ${\mathcal I}-$expansive index too.
\\
We say $(T,(X,\mathcal{K}))$ is expansive if it is $\{\varnothing\}$-expansive
(see e.g.,~\cite{khodam}).
\end{note}
%%%%%%%%%%%%%%%%%%%%%%%%%%%%%%%%%%%%
\noindent If $\mathcal{I}$ and $\mathcal{J}$ are two ideals on $T$ with $\mathcal{J}\subseteq \mathcal{I}$,
it is evident that $\mathcal{I}-$expansivity of $(T,X)$ leads to
$\mathcal{J}-$expansivity of $(T,X)$. In particular, $\mathcal{I}-$expansivity of $(T,X)$
leads to expansivity of $(T,X)$. However, the following counterexample shows that
$\mathcal{I}-$expansivity and expansivity are not equivalent.
%%%%%%%%%%%%%%%%%%%%%%%%%%%%%%%%%%%%
\begin{counterexample}\label{salam217}
Let $X:=(0,+\infty)$ with induced metric $d$ from Euclidean line $\mathbb R$,
and uniform structure $\mathcal{K}:=\mathcal{K}_d$,
for each $n\in\mathbb{Z}$ consider $f_n:X\to X$ with $f_n(x)=x^n(x\in X)$,
also let $\mathcal{I}:=\mathcal{P}(\{f_{-n}:n\geq1\})$ and $T=\{f_n:n\in\mathbb{Z}\}$. Then
the transformation semigroup $(T,X)$ is expansive and it is not $\mathcal{I}-$expansive.
\end{counterexample}
%%%%%%%%%%%%%%%%%%%%%%%%%%%%%%%%%%%%
\begin{proof}
Note that $\psi:=\{(x,y)\in X\times X:|x-y|<1/2\}\in\mathcal{K}$.
Consider distinct $x,y\in X$ we have the following cases:
\begin{itemize}
\item[i.] $x,y\geq1$. We may suppose $x>y\geq1$,
	there exists $n\geq1$ such that $(x/y)^n>2$, thus
	$d(f^n(x),f^n(y))=|f^n(x)-f^n(y)|=x^n-y^n>y^n\geq1$,
	so $(f_n(x),f_n(y))\notin\psi$.
\item[ii.] $x,y\leq1$. In this case $x^{-1},y^{1}\geq1$, so using (i)
	there exists $m\geq1$ such that
	$(f_{-m}(x),f_{-m}(y))=(f_m(x^{-1}),f_m(y^{-1}))\notin\psi$.
\item[iii.] $\min(x,y)<1<\max(x,y)$. We may suppose $x<1<y$,
	there exists $n\geq1$ such that $y^n>2>1>x\geq x^n$, thus
	$d(f^n(x),f^n(y))=|f^n(x)-f^n(y)|=y^n-x^n>1$,
	so $(f_n(x),f_n(y))\notin\psi$.
\end{itemize}
Using (i), (ii), (iii), $(T,X)$ is expansive with expansive index $\psi$.
\\
If $(T,X)$ is $\mathcal{I}-$expansive with $\mathcal{I}-$expansive $\mu$ index
then there exists $\delta\in(0,1)$ such that $\theta:==\{(x,y)\in X\times X:|x-y|<\delta\}\subseteq\mu$.
there exists $f_m\in T\setminus\{f_{-n}:n\geq1\}$ such that $(f_m(\delta/2),f_m(\delta/4))\notin\mu$ thus
$m\geq1$ and $(f_m(\delta/2),f_m(\delta/4))\notin\theta$. Therefore
$\delta>d(f_m(\delta/2),f_m(\delta/4))=(\delta/2)^m-(\delta/4)^m$ which is a contradiction,
hence $(T,X)$ is not $\mathcal{I}-$expansive.
\end{proof}
%%%%%%%%%%%%%%%%%%%%%%%%%%%%%%%%%%%%
\begin{lemma}\label{salam230}
Suppose  $(T,(X,\mathcal{K}))$ is
${\mathcal I}-$expansive with ${\mathcal I}-$expansive index $\psi$, also
$\beta$ is an index with $\beta^{-1}\circ\beta\subseteq\psi$, then each sequence $(x_t)_{t\in T}$ in $X$,
has at most one $\beta-$trace.
\end{lemma}
%%%%%%%%%%%%%%%%%%%%%%%%%%%%%%%%%%%%
\begin{proof}
Suppose $x,y\in X$ are $\beta-$traces of $(x_t)_{t\in T}$, then
\[\forall t\in T\:((tx,x_t)\in\beta\wedge(x_t,ty)\in\beta^{-1})\:, \]
so
$(tx,ty)\in \beta^{-1}\circ\beta\subseteq\psi$ for each $t\in T$,
which leads to $x=y$ since $\psi$ is an ${\mathcal I}-$expansive index of transformation
semigroup $(T,(X,\mathcal{K}))$.
\end{proof}
%%%%%%%%%%%%%%%%%%%%%%%%%%%%%%%%%%%%
\begin{corollary}\label{salam240}
For nonempty subset $A$ of $T$,
suppose $(T,(X,\mathcal{K}))$
is ${\mathcal I}-$expansive with
${\mathcal I}-$expansive index $\psi$ and has shadowing property with respect to $A$, then there exist
$\theta,\lambda\in\mathcal{K}$ such that
$\lambda^{-1}\circ\lambda\subseteq\psi$ and each
$\theta-$pseudo orbit with respect to $A$
has a $\lambda-$trace.
By Lemma~\ref{salam230}
each $\theta-$pseudo orbit with respect to $A$ has a unique $\lambda-$trace.
\end{corollary}
%%%%%%%%%%%%%%%%%%%%%%%%%%%%%%%%%%%%
\begin{lemma}\label{salam250}
Suppose $X$ is compact and $(T,(X,\mathcal{K}))$ is ${\mathcal I}-$expansive with closed ${\mathcal I}-$expansive index $\psi$, also
$\beta$ is an open index, then for all $E\in\mathcal{I}$
there exists nonempty finite subset $F$ of $T\setminus E$ such that
\[\forall x,y\in X\SP (F(x,y)=\{(tx,ty):t\in F\}\subseteq\psi\Rightarrow(x,y)\in\beta)\:.\]
\end{lemma}
%%%%%%%%%%%%%%%%%%%%%%%%%%%%%%%%%%%%
\begin{proof}
Suppose there exists $E\in\mathcal{I}$ such that for each nonempty finite subset $F$ of
$T\setminus E$ there exists $x_F,y_F\in X$ such that $F(x_F,y_F)\subseteq\psi$ and $(x_F,y_F)\notin\beta$.
Consider directed set $(\mathcal{P}_{fin}(T\setminus E)\setminus\{\varnothing\},\subseteq)$, then in compact space $X\times X$ the net
$\{(x_F,y_F)\}_{F\in \mathcal{P}_{fin}(T\setminus E)\setminus\{\varnothing\}}$ has a convergent subnet like
$\{(x_{F_\lambda},y_{F_\lambda})\}_{\lambda\in\Lambda}$ to a point of $X\times X$ like $(z,w)$.
For each $\lambda\in\Lambda$, $(x_{F_\lambda},y_{F_\lambda})\in (X\times X)\setminus\beta$. Thus
$(z,w)\in\overline{(X\times X)\setminus\beta}=(X\times X)\setminus\beta$, i.e.
\begin{equation}\label{eq10}
(z,w)\notin \beta\:.
\end{equation}
On the other hand for each $t\in T\setminus E$, there exists $\lambda_t\in\Lambda$ such that $\{t\}\subseteq F_{\lambda_t}$.
For each $\lambda\geq\lambda_t$ we have $\{t\}\subseteq F_{\lambda_t}\subseteq F_{\lambda}$, thus
$(tx_{F_\lambda},ty_{F_\lambda})\in F_\lambda(x_{F_\lambda},y_{F_\lambda})\subseteq\psi$, therefore
$(tz,tw)\in\overline{\psi}=\psi$ (since $\{(tx_{F_\lambda},ty_{F_\lambda})\}_{\lambda\in\Lambda,\lambda\geq\lambda_t}$
is a net in $\psi$ converges to $(tz,tw)$). Since $(tz,tw)\in\psi$ for each $t\in T\setminus E$ and $\psi$ is an ${\mathcal I}-$expansive index of $X$, $z=w$
which is in contradiction with \ref{eq10}.
\end{proof}
%%%%%%%%%%%%%%%%%%%%%%%%%%%%%%%%%%%%
\section{Interaction between $\mathcal{I}-$shadowing, $\mathcal{I}-$expansiveness and $\mathcal{I}-$topological stability: a classic approach}
\noindent In this section consider uniform space $(X,\mathcal{K})$,
equip $C(X)$ and $Homeo(X)$ with induced uniform structure and induced topology of
$(F(X),\mathcal{H}_{\mathcal K})$.
\\
For $A\subseteq T$ and $\theta\in\mathcal{K}$ let:
\[\mathcal{C}(A,T,\theta):=\{((f_t)_{t\in T},(g_t)_{t\in T})\in F(X)^T\times F(X)^T:\forall t\in A\:\:(f_t,g_t)\in B_\theta\}\:,\]
then
$\mathfrak{U}_\mathcal{I}:=\{\alpha\subseteq F(X)^T\times F(X)^T:$ there exist $\lambda\in\mathcal{K}$ and $B\in\mathcal{I}$ such that
$\mathcal{C}(B,T,\lambda)\subseteq\alpha\}$ is uniform structure on $F(X)^T$.
For semigroup $T$ equip $C(X)^T$ and
$Homeo(X)^T$ uniform topology induced from $(F(X)^T,\mathfrak{U}_\mathcal{I})$, then
we may consider $Act(T,X)$ as a subspace of $C(X)^T$, if $T$ is a group then
we may consider $Act(T,X)$ as a subspace of $Homeo(X)^T$.
%%%%%%%%%%%%%%%%%%%%%%%%%%%%%%%%%%%%
\begin{remark}\label{ali10}
Let's compare $\mathfrak{U}_\mathcal{I}$ for ideals $\mathcal{I}=\{\varnothing\},{\mathcal{P}_{fin}(T)},{\mathcal{P}(T)}$.
\begin{itemize}
\item For each $\theta\in\mathcal{K}$, $\mathcal{C}(\varnothing,T,\theta)=F(X)^T\times F(X)^T$, thus uniform topology induced from 
	$\mathfrak{U}_{\{\varnothing\}}$ on $F(X)^T$ is trivial topology $\{F(X)^T,\varnothing\}$.
\item Uniform structure $\mathfrak{U}_{\mathcal{P}_{fin}(T)}$
	on $F(X)^T$ is compatible with its pointwise convergence (product) topology.
\item Uniform structure $\mathfrak{U}_{\mathcal{P}(T)}$
	on $F(X)^T$ is compatible with its uniform convergence topology
	(use a similar method described in Note~\ref{taha15}).
	Moreover $\mathfrak{U}_{\mathcal{P}(T)}=\{\alpha\subseteq F(X)^T\times F(X)^T:$ 
	there exists $\lambda\in\mathcal{K}$  such that
$\mathcal{C}(T,T,\lambda)\subseteq\alpha\}$.
\end{itemize}
\end{remark}
%%%%%%%%%%%%%%%%%%%%%%%%%%%%%%%%%%%%
%%%%%%%%%%%%%%%%%%%%%%%%%%%%%%%%%%%%
\begin{remark}\label{ali20}
If $X$ is discrete with $\Delta_X\in\mathcal{K}$, then 
\begin{itemize}
\item for all $A\in\mathcal{I}$ we have
$\mathcal{C}(A,T,\Delta_X)=\{((f_t)_{t\in T},(g_t)_{t\in T})\in F(X)^T\times F(X)^T:\forall t\in A\:\:f_t=g_t\}$
moreover $\mathfrak{U}_{\mathcal I}=\{\alpha\subseteq F(X)^T\times F(X)^T:$ 
	there exists $A\in\mathcal{I}$  such that
$\mathcal{C}(A,T,\Delta_X)\subseteq\alpha\}$,
\item
 $\mathcal{C}(T,T,\Delta_X)=\Delta_{F(X)^T}$, hence by Remark~\ref{ali10}, 
 $\mathfrak{U}_{\mathcal{P}(T)}=\{\alpha\subseteq F(X)^T\times F(X)^T:$ 
$\Delta_{F(X)^T}\subseteq\alpha\}$, which induced discrete topology on $F(X)^T$.
\end{itemize}
\end{remark}
%%%%%%%%%%%%%%%%%%%%%%%%%%%%%%%%%%%%
\begin{definition}\label{salam220}
We say
$\mathfrak{X}\in Act(T,(X,\mathcal{K}))$ is a topological stable action 
of semigroup $T$ on $X$ modulo ideal $\mathcal I$
(or briehly ${\mathcal I}-$topological stable action of semigroup $T$
on $X$) if for each $\theta\in\mathcal{K}$ there exists open neighbourhood $U$ of $\mathfrak{X}$ in $(F(X)^T,\mathfrak{U}_\mathcal{I})$ such that for each $\mathfrak{Z}\in U
\cap Act(T,X)$
there exists homomorphism $f:(T,X,\mathfrak{Z})\to(T , X , \mathfrak{X})$
(i.e. $f:X\to X$ is continuous and $f(tx)=tf(x)$ for all $t\in T,x\in X$) with
$(f,id_X)\in B_\theta$.
\\
We say $\mathfrak{X}$ is topological stable if it is $\mathcal{P}_{fin}(T)-$topological
stable.
\end{definition}
%%%%%%%%%%%%%%%%%%%%%%%%%%%%%%%%%%%%
\begin{theorem}\label{ali30}
Suppose $X$ is discrete with $\Delta_X\in\mathcal{K}$. $\mathfrak{X}\in Act(T,X)$ is $\mathcal{I}-$topological stable
if and only if $\{\mathfrak{X}\}$ is an isolated point of $Act(T,X)$ (with induced uniform topology of  $(F(X)^T,\mathfrak{U}_\mathcal{I})$).
\end{theorem}
%%%%%%%%%%%%%%%%%%%%%%%%%%%%%%%%%%%%
\begin{proof}
Suppose $\mathfrak{X}$ is an isolated point of  $Act(T,X)$ in $(F(X)^T,\mathfrak{U}_\mathcal{I})$, then there exists open neighbourhood of
$\mathfrak{X}$ in $(F(X)^T,\mathfrak{U}_\mathcal{I})$ such that $U\cap Act(T,X)=\{\mathfrak{X}\}$. For all $\theta\in\mathcal{K}$
and $\mathfrak{Z}\in U\cap Act(T,X)$ we have $\mathfrak{Z}=\mathfrak{X}$ and 
$id_X:(T,X,\mathfrak{Z})=(T,X,\mathfrak{X})\to(T , X , \mathfrak{X})$ with $(id_X,id_X)\in B_\theta$.
Therefore $\mathfrak{X}$ is $\mathcal{I}-$topological stable
\\
Now suppose $\mathfrak{X}'\in Act(T,X)$ is $\mathcal{I}-$topological stable. Choose open neighbourhood 
$V$ of $\mathfrak{X}'$ in $(F(X)^T,\mathfrak{U}_\mathcal{I})$ such that for each $\mathfrak{Z}\in V\cap Act(T,X)$ there exists
homomorphism $f_\mathfrak{Z}:(T,X,\mathfrak{Z})\to(T , X , \mathfrak{X}')$ with $(f_\mathfrak{Z},id_X)\in B_{\Delta_X}$,
thus $f_\mathfrak{Z}=id_X$ which leads to $\mathfrak{X}'=\mathfrak{Z}$. Hence $V\cap Act(T,X)=\{\mathfrak{X}'\}$
and $\mathfrak{X}'$ is an isolated point of $Act(T,X)$ with induced uniform topology of  $(F(X)^T,\mathfrak{U}_\mathcal{I})$.
\end{proof}
%%%%%%%%%%%%%%%%%%%%%%%%%%%%%%%%%%%%
\noindent By the following example topological stability and $\mathcal{I}-$topological stability are not equivalent in general.
%%%%%%%%%%%%%%%%%%%%%%%%%%%%%%%%%%%%
\begin{example}
Equip $X=\mathbb{Z}_2=T$ with discrete topology (where $\mathbb{Z}_2$ is qoutient group $\frac{\mathbb{Z}}{2\mathbb{Z}}$). Then
$Act (T,X)$ has two elements and:
\begin{itemize}
\item by Remark~\ref{ali20} and Theorem~\ref{ali30} all elements of $Act (T,X)$ are $\mathcal{P}(T)-$topological stable
(and topological stable),
\item by Remark~\ref{ali20} and Theorem~\ref{ali30}, $Act (T,X)$ does not have any $\{\varnothing\}-$topological stable element.
\end{itemize}
\end{example}
%%%%%%%%%%%%%%%%%%%%%%%%%%%%%%%%%%%%
\begin{lemma}\label{salam260}
Suppose $\theta\in\mathcal{K}$, $A\subseteq T$, $x\in X$ and $\mathfrak{X},\mathfrak{Z}\in Act(T,X)
(\subseteq C(X)^T)$ with $(\mathfrak{X},\mathfrak{Z})\in \mathcal{C}(A,T,\theta)$. Then $(\mathfrak{Z}^tx)_{t\in T}$
is a $\theta-$pseudo orbit with respect to $A$ in $(T,X,\mathfrak{X})$.
\end{lemma}
%%%%%%%%%%%%%%%%%%%%%%%%%%%%%%%%%%%%
\begin{proof}
For each $a\in A$, $(\mathfrak{X}^a,\mathfrak{Z}^a)\in B_\theta$, thus for each $y\in X$, $(\mathfrak{X}^ay,\mathfrak{Z}^ay)\in\theta$.
In particular for each $t\in T$, $(\mathfrak{X}^a\mathfrak{Z}^tx,\mathfrak{Z}^a\mathfrak{Z}^tx)\in\theta$,
i.e. $(\mathfrak{X}^a\mathfrak{Z}^tx,\mathfrak{Z}^{at}x)\in\theta$, which leads to the desired result since $a\in A$ is arbitrary.
\end{proof}
%%%%%%%%%%%%%%%%%%%%%%%%%%%%%%%%%%%%
%%%%%%%%%%%%%%%%%%%%%%%%%%%%%%%%%%%%
\begin{theorem}\label{salam270}
Every ${\mathcal I}-$expansive, compact Hausdorff transformation (semi)group with ${\mathcal I}-$shadowing property, is ${\mathcal I}-$topological stable.
\end{theorem}
%%%%%%%%%%%%%%%%%%%%%%%%%%%%%%%%%%%%
\begin{proof}
In uniform space $(X,\mathcal{K})$ consider ${\mathcal I}-$expansive action $\mathfrak{X}\in Act(T,X)$ with ${\mathcal I}-$expansive index $\psi\in\mathcal{K}$ we may suppose
$\psi$ is a closed index
also suppose $(T,X,\mathfrak{X})$ has ${\mathcal I}-$shadowing property.
Consider $\mu\in\mathcal{K}$ there exists nonempty set $A\in\mathcal{I}$ such that
$(T,X,\mathfrak{X})$ has shadowing property with respect to $A$.
Choose open index $\beta^{-1}=\beta\in\mathcal{K}$ such that $\beta\circ\beta\circ\beta\subseteq \psi\cap\mu$. There exists $\theta\in\mathcal{K}$ such that
each $\theta-$pseudo orbit with respect to $A$ in $(T,X,\mathfrak{X})$ has a $\beta-$trace. By Corollary~\ref{salam240}
each $\theta-$pseudo orbit with respect to $A$ in $(T,X,\mathfrak{X})$ has a unique $\beta-$trace.
There exists open neighbourhood $U$ of $\mathfrak{X}$ (in $F(X)^T$) such that $U\times U\subseteq \mathcal{C}(A,T,\theta)$.
Choose $\mathfrak{Z}\in U\cap Act(T,X)$, so $(\mathfrak{X},\mathfrak{Z})\in \mathcal{C}(A,T,\theta)$.
For each $x\in X$, by Lemma~\ref{salam260}, $(\mathfrak{Z}^tx)_{t\in T}$
is a $\theta-$pseudo orbit with respect to $A$ in $(T,X,\mathfrak{X})$ so it has a unique $\beta-$trace like $f(x)\in X$,
i.e.,
\begin{equation}\label{eq20}
\forall x\in X\:\: \forall t\in T\:\: (\mathfrak{X}^tf(x),\mathfrak{Z}^tx)\in\beta\:.
\end{equation}
In particular $(f(x),x)=(\mathfrak{X}^ef(x),\mathfrak{Z}^ex)\in\beta$ for all $x\in X$, hence
\begin{equation}\label{eq25}
(f,id_X)\in B_\beta\subseteq B_\mu\:.
\end{equation}
for each $x\in X$, $h\in T$, and $t\in T$ by \ref{eq20} we have
\[(\mathfrak{X}^tf(\mathfrak{Z}^hx),\mathfrak{Z}^t\mathfrak{Z}^hx)=(\mathfrak{X}^tf(\mathfrak{Z}^hx),\mathfrak{Z}^{th}x)\in\beta\:,
(\mathfrak{X}^{th}f(x),\mathfrak{Z}^{th}x)\in\beta\]
therefore
\begin{equation}\label{eq30}
(\mathfrak{X}^{t}\mathfrak{X}^hf(x),\mathfrak{X}^tf(\mathfrak{Z}^hx))=(\mathfrak{X}^{th}f(x),\mathfrak{X}^tf(\mathfrak{Z}^hx))\in\beta\circ\beta\subseteq\psi\:.
\end{equation}
For each $\eta\in\mathcal{K}$, by Lemma~\ref{salam250}
there exists nonempty finite subset $F$ of $T$ such that for all $z,w\in X$ with
$\{(\mathfrak{X}^tz,\mathfrak{X}^tw):t\in F\}\subseteq\psi$ we
have $(z,w)\in\eta$. By~\ref{eq30}, $\{(\mathfrak{X}^{t}\mathfrak{X}^hf(x),\mathfrak{X}^tf(\mathfrak{Z}^hx)):t\in F\}\subseteq\psi$
which leads to
\[(\mathfrak{X}^hf(x),f(\mathfrak{Z}^hx))\in \eta\]
Hence $(\mathfrak{X}^hf(x),f(\mathfrak{Z}^hx))\in\bigcap\{\eta:\eta\in\mathcal{K}\}=\Delta_X$, so
\begin{equation}\label{eq35}
\forall h\in T\:\:\forall x\in X\:\:
\mathfrak{X}^hf(x)=f(\mathfrak{Z}^hx)\:.
\end{equation}
For each $\xi\in\mathcal{K}$, by Lemma~\ref{salam250}
there exists nonempty finite subset $E$ of $T$ such that for all $z,w\in X$ with
$\{(\mathfrak{X}^tz,\mathfrak{X}^tw):t\in E\}\subseteq\psi$ we
have $(z,w)\in\xi$. For each $t\in E$, $\mathfrak{Z}^t:(X,\mathcal{K})\to (X,\mathcal{K})$ is continuous, hence it is
uniformly continuous by compactness of $ (X,\mathcal{K})$, by finiteness of $E$ there exists $\phi\in\mathcal{K}$ such that
\begin{equation}\label{eq40}
\{(\mathfrak{Z}^tz,\mathfrak{Z}^tw):t\in E, (z,w)\in\phi\}\subseteq\xi\cap\beta\:.
\end{equation}
For each $(z,w)\in \phi$ and $t\in E$ we have:
\begin{itemize}
\item $(\mathfrak{Z}^tz,\mathfrak{Z}^tw)\in\beta$ by \ref{eq40},
\item $(f(\mathfrak{Z}^tz),\mathfrak{Z}^tz),(f(\mathfrak{Z}^tw),\mathfrak{Z}^tw)\in \beta$ by \ref{eq25},
\end{itemize}
hence $(\mathfrak{X}^tf(z),\mathfrak{X}^tf(w))=(f(\mathfrak{Z}^tz),f(\mathfrak{Z}^tw))\in\beta\circ\beta\circ\beta\subseteq\psi$ by \ref{eq35}.
\\
Since $\{(\mathfrak{X}^tf(z),\mathfrak{X}^tf(w)):t\in E,(z,w)\in\phi\}\subseteq \psi$ we have
$\{(f(z),f(w)):(z,w)\in\phi\}\subseteq\xi$, therefore $f:(X,\mathcal{K})\to(X,\mathcal{K})$ is uniformly continuous, so it is continuous.
\end{proof}
%%%%%%%%%%%%%%%%%%%%%%%%%%%%%%%%%%%%
%%%%%%%%%%%%%%%%%%%%%%%%%%%%%%%%%%%%
%%%%%%%%%%%%%%%%%%%%%%%%%%%%%%%%%%%%
%%%%%%%%%%%%%%%%%%%%%%%%%%%%%%%%%%%%
%%%%%%%%%%%%%%%%%%%%%%%%%%%%%%%%%%%%
%%%%%%%%%%%%%%%%%%%%%%%%%%%%%%%%%%%%
%%%%%%%%%%%%%%%%%%%%%%%%%%%%%%%%%%%%
%%%%%%%%%%%%%%%%%%%%%%%%%%%%%%%%%%%%

%
%%%%%%%%%%%%%%%%%%%%%%%%%%%%%%%%%%%%
 %%%%%%%%%%%%%%%%%%%%%%%%%%%%%%%%%%%%
\noindent \noindent {\small {\bf Fatemah Ayatollah Zadeh Shirazi}, Faculty
of Mathematics, Statistics and Computer Science, College of
Science, University of Tehran, Enghelab Ave., Tehran, Iran
\linebreak (f.a.z.shirazi@ut.ac.ir)}
\\
{\small {\bf Elaheh Hakimi}, Faculty of Mathematics, Statistics
and Computer Science, College of Science, University of Tehran,
Enghelab Ave., Tehran, Iran (elaheh.hakimi@gmail.com)}
\\
{\small {\bf Arezoo Hosseini},
Department of Mathematics, Education Farhangian University,
\linebreak
P.~O.~Box 14665--889, Tehran, Iran
(a.hosseini@cfu.ac.ir)}
\\
{\small {\bf Khosro Tajbakhsh}, Department of Mathematics, Faculty of Mathematical Sciences, Tarbiat Modares
University, Tehran 14115-134, Iran (khtajbakhsh@modares.ac.ir)}
%\[\underline{\SP\SP\SP\SP\SP\SP\SP\SP\SP\SP\SP\SP\SP\SP\SP\SP}\]
 \end{document}